\documentclass[11pt]{article}
\usepackage{amsfonts,amsmath,tikz}
\usepackage{epsfig}
\usepackage{latexsym}
\usepackage{color}
\usepackage{epsf,amssymb}
\usepackage{enumitem}

\usepackage{amsthm}
\usepackage{yfonts}
\usepackage{hyperref}
\usepackage{ifthen}

\usepackage{cite}

\usetikzlibrary{shapes,positioning}
\usetikzlibrary{calc}

\usepackage{float} 
\floatstyle{plaintop}

\newtheorem{theorem}{Theorem}
\newtheorem{lemma}[theorem]{Lemma}

\newtheorem{corollary}[theorem]{Corollary}
\newtheorem{proposition}[theorem]{Proposition}

\newcommand{\cind}{c_{\rm ind}}

\newcommand{\gammaenaL}{\gamma_1^{\mathrm L}}
\newcommand{\gammakL}{\gamma_k^{\mathrm L}}

\tikzset{My Style/.style={draw, circle, fill=black,scale=0.5}} 

\title{Structural aspects of $1$-limited domination in cubic graphs}

\author{
Goran Radi\'c \\
\small \it University of Maribor, FNM, Koro\v ska cesta 160, 2000, Maribor, Slovenia
\\
\small \tt radic.goran@gmail.com
\and
Gordana Radi\'{c} \\
\small \it University of Maribor, FEECS, Koro\v ska cesta 46, 2000, Maribor, Slovenia\\
\small \it Institute of Mathematics, Physics and Mechanics,  Jadranska ul.~19, 1000 Ljubljana, Slovenia \\
\small \tt gordana.radic@um.si
\and
 Aleksandra Tepeh \\
\small \it University of Maribor, FEECS, Koro\v ska cesta 46, 2000, Maribor, Slovenia \\
\small \it 
Institute of Mathematics, Physics and Mechanics,  Jadranska ul.~19, 1000 Ljubljana, Slovenia \\
\small \it University of Maribor, FNM, Koro\v ska cesta 160, 2000, Maribor, Slovenia\\
\small \tt aleksandra.tepeh@um.si  \\
}

  \date{}    

\begin{document}
\maketitle

\begin{abstract}
A set $D\subseteq V(G)$ is a $1$-limited dominating set of a graph $G$ if every vertex outside $D$ has a neighbor in $D$ and every vertex of $D$ has at most one neighbor outside $D$. The minimum cardinality among all $1$-limited dominating sets of $G$ is denoted by $\gammaenaL(G)$. 

For every cubic graph $G$ of order $n$, there exists a nonnegative integer $t$ such that $\gammaenaL(G)=\frac n2+t$. We establish structural properties of $1$-limited dominating sets in terms of $t$ and use them to study the induced cycles vertex number, that is, the maximum number of vertices inducing a $2$-regular subgraph. As applications, we determine the exact $1$-limited domination numbers of generalized Blanu\v{s}a snarks and flower snarks.
\end{abstract}


\medskip


\section{Introduction and preliminaries}

Throughout the paper, $G$ denotes a finite simple graph with vertex set $V(G)$ and edge set $E(G)$. For a vertex $v\in V(G)$, let $N_G(v)$ denote its neighborhood, and for a set $S\subseteq V(G)$, let $\overline S=V(G)\setminus S$.

Domination is one of the most extensively studied concepts in graph theory, owing to both its theoretical significance and its applications in areas such as network design, facility location, and theoretical computer science. As the literature on domination and its many variants is vast, we refer the interested reader to several monographs for an introduction and further references (see, e.g., \cite {HHH1,HHH2,HHH3}).

Recall that a \emph{dominating set} of $G$ is a set $D\subseteq V(G)$ such that every vertex in $\overline{D}$ has a neighbor in $D$. In this paper, we consider a recently introduced refinement called \emph{$k$-limited domination}. For a positive integer $k$, a dominating set $D$ of $G$ is a \emph{$k$-limited dominating set} if $|N_G(v)\setminus D|\leq k$ for every $v\in D$. We refer to this additional requirement as the \emph{$k$-limited constraint}. The minimum cardinality among all $k$-limited dominating sets of $G$ is called the \emph{$k$-limited domination number} and is denoted by $\gammakL(G)$. Since the $k$-limited domination number is additive over connected components, the problem reduces naturally to connected graphs. 

The concept of $k$-limited domination was introduced in~\cite{podgorica1}, where, among other results, general bounds of $k$-limited domination number and connections with graph packings were established. These bounds were further refined in~\cite{G}, using connections with capacitated domination and limited packing to obtain improved lower and upper bounds. Further developments include the study of $k$-limited domination in Cartesian products~\cite{S} and Sierpi\'nski graphs~\cite{D}.

The $k$-limited domination model is related to several other domination concepts, in particular to $(1,s)$-domination. For a positive  integer $s\geq1$, a dominating set $D\subseteq V(G)$ of a graph $G$ with minimum degree at least $s$ is called a \emph{$(1,s)$-dominating set} if every vertex of $D$ has at least $s$ neighbors in $D$. The minimum cardinality among all such sets is denoted by $\gamma_{1,s}(G)$, see~\cite{fakhran}.

A key observation for the present study concerns regular graphs. If $G$ is $r$-regular and $1\leq k<r$, then, for every set $D\subseteq V(G)$ and every $v\in D$,
$|N_G(v)\setminus D|\leq k$
if and only if
$|N_G(v)\cap D|\geq r-k$.
Consequently,
$\gamma_k^{\mathrm L}(G)=\gamma_{1,r-k}(G)$.
In particular, if $G$ is cubic, then
\[
\gammaenaL(G)=\gamma_{1,2}(G).
\]
Thus, $1$-limited domination in cubic graphs can equivalently be viewed as $(1,2)$-domination.

The \emph{induced cycles vertex number} of a graph $G$, denoted by $\cind(G)$, is the maximum cardinality of a set $S\subseteq V(G)$ such that $G[S]$ is $2$-regular. Henning et al.~\cite{Henning} proved that
$\cind(G)\geq \frac n4+\frac12$
for every cubic graph $G$ of order $n$, and that
$\cind(G)>\frac{13n}{20}$
when $G$ is additionally claw-free. They also conjectured that
\begin{equation}
\label{eq:cind-conj}
c_{\mathrm{ind}}(G)\geq\frac n2
\end{equation}
for every cubic graph $G$.
On the other hand, as shown in~\cite{podgorica1,fakhran}, every cubic graph $G$ of order $n$ satisfies
\begin{equation}
\label{eq}
\gammaenaL(G)\geq\frac n2.
\end{equation}
Since every cubic graph has even order, there exists a nonnegative integer $t$ such that $$\gammaenaL(G)=\frac{n}{2}+t.$$
This suggests studying whether the parameter $t$, together with the structure of a minimum $1$-limited dominating set, can provide information about induced $2$-regular subgraphs.

Cubic graphs form one of the most prominent classes in graph theory. Their importance stems in part from the fact that several difficult problems and conjectures can be reduced to, or are naturally formulated for, cubic graphs. A classic example is the Four Color Theorem: Tait showed that it is equivalent to the statement that every bridgeless planar cubic graph admits a proper $3$-edge-coloring~\cite{Tait}. More generally, several important conjectures reduce to the cubic case, and for many of them the $3$-edge-colorable case is straightforward; see~\cite{snarki}. Consequently, the main difficulty often lies in cubic graphs that are not $3$-edge-colorable, which naturally leads to the study of snarks.

The precise definition of a snark varies slightly in the literature. In this paper, we call a connected bridgeless cubic graph of girth at least five that is not $3$-edge-colorable a \emph{snark}. The Petersen graph is the classical first example, while Isaacs' constructions from 1975 produced the first infinite families of snarks~\cite{Isaacs}. Many further families are now known; for historical background, we refer to~\cite{Bradley}. Despite their central role in problems concerning edge-colorability and related conjectures, snarks have only relatively recently been studied from the perspective of domination parameters.

In 2016, Zhao and Liu determined the connected domination and tree domination numbers of Goldberg and flower snarks \cite{Zhao}. Several variants of Roman domination on snarks were studied in \cite{maxim, Luiz}, while weakly convex and convex domination of flower snarks were considered in \cite{kratica}. More recently, Burdett, Haythorpe, and Newcombe determined the exact values of several further domination parameters for flower snarks \cite{Burdett}, while domination and independent domination in generalized Blanu\v{s}a and Loupekine snarks were studied in~\cite{amc}.

The study of domination on snarks was extended to $1$-limited domination in~\cite{GRAT-gold}. It was shown that determining whether a cubic graph has a $1$-limited dominating set of a given size is NP-complete even for $2$-connected planar cubic graphs, and the exact $1$-limited domination number of the Goldberg snarks was determined using a block-based analysis with local forcing and global compensation. We 
note that the same type of direct block-by-block analysis can be adapted to flower snarks; in the present paper, however, we develop general structural results for $1$-limited dominating sets in cubic graphs and use them to determine the exact values for generalized Blanu\v{s}a snarks and flower snarks.

The remainder of the paper is organized as follows. In Section 2, we describe the structure of cubic graphs $G$ of order $n$ with $\gammaenaL(G)=\frac n2+t$, where $t$ is a nonnegative integer, and discuss its relationship with induced $2$-regular subgraphs. In Section 3, we apply some of these results to determine the exact $1$-limited domination number of generalized Blanu\v{s}a snarks and flower snarks. We conclude the paper in Section 4 with a further discussion of the relationship between the $1$-limited domination number and the induced cycles vertex number.

%
\section{$1$-limited domination and induced cycles}
%

We begin by developing structural properties of $1$-limited dominating sets in cubic graphs. Throughout this section, we use the following notation. For a positive integer $k$, let $[k]=\{1,\ldots,k\}$. For a vertex $v\in V(G)$, its degree is denoted by $\deg_G(v)$. For a set $S\subseteq V(G)$, we denote by $G[S]$ the subgraph of $G$ induced by $S$, while $G[S,\overline S]$ denotes the bipartite spanning subgraph consisting of all edges with one endpoint in $S$ and the other in $\overline S$. The complete bipartite graph with partite sets of cardinalities $m$ and $n$ is denoted by $K_{m,n}$. For $n\geq2$, we call $K_{1,n}$ a \emph{star} and refer to its vertex of degree $n$ as its center. 

Let $D$ be a $1$-limited dominating set of a cubic graph $G$. Our first goal is to describe how the $1$-limited constraint is reflected in the bipartite spanning subgraph $G[D,\overline D]$.

To describe this relationship, for each $i\in\{0,1,2,3\}$, let $d_i$ denote the number of vertices of degree $i$ in $G[D,\overline{D}]$. That is,
$$
d_i=|\{v\in V(G):\deg_{G[D,\overline D]}(v)=i\}|.
$$
Since $D$ is a dominating set, every vertex of $\overline D$ has at least one neighbor in $D$, while the $1$-limited constraint implies that every vertex of $D$ has at most one neighbor in $\overline D$. Taken together, these observations imply that all $d_0$ isolated vertices of $G[D,\overline D]$ belong to $D$, whereas all its vertices of degree two or three belong to $\overline D$. 

\begin{theorem}\label{thm}
Let $D$ be a $1$-limited dominating set of a cubic graph $G$ of order $n$.
Let $t$ be the nonnegative integer such that
$|D|=\frac{n}{2}+t$. Then
\begin{enumerate}
\item[$\mathrm{(i)}$] $d_0+d_2+2d_3=2t$;
\item[$\mathrm{(ii)}$] $d_0$ and $d_2$ are even;
\item[$\mathrm{(iii)}$] $G[D,\overline D]$ is a disjoint union of $d_0$ isolated vertices, $d_2$ copies of $K_{1,2}$ and $d_3$ copies of $K_{1,3}$, while all its remaining components are isomorphic to $K_2$. The centers of all star components belong to $\overline D$.
\end{enumerate}
\end{theorem}

\begin{proof}
Let $D\subseteq V(G)$ be a $1$-limited dominating set with
$|D|=\frac n2+t$ and $|\overline D|=\frac n2-t$. Since exactly $d_0$ vertices of $D$ are isolated in $G[D,\overline D]$, while every other vertex of $D$ is of degree one, it follows that
$$
|E(G[D,\overline D])|
=|D|-d_0
=\tfrac n2+t-d_0.
$$
On the other hand, every vertex of $\overline{D}$ has degree at least one in $G[D,\overline D]$. More precisely, each of the $d_2$ vertices of degree two contributes one additional edge, while each of the $d_3$ vertices of degree three contributes two additional edges. Hence,
$$
|E(G[D,\overline D])|
=|\overline D|+d_2+2d_3
=\tfrac n2-t+d_2+2d_3.
$$
Equating the two expressions for $|E(G[D,\overline{D}])|$ gives $d_0+d_2+2d_3=2t$, which proves~(i). Furthermore,
$$
\sum_{v\in D}\deg_{G[D]}(v)
=3|D|-|E(G[D,\overline D])|
=3|D|-(|D|-d_0)
=2|D|+d_0.
$$
Consequently, $d_0$ must be even. Moreover, as $d_2=2t-2d_3-d_0$ by (i), it yields that $d_2$ is even as well. 

It remains to prove the decomposition stated in~(iii). First observe that, since $D$ is a dominating set, every isolated vertex of $G[D,\overline D]$ belongs to $D$, and so every nontrivial component $C$ of $G[D,\overline D]$ has at least one vertex of $\overline D$.  We claim that it contains exactly one such vertex. Otherwise, if $C$ contained two vertices of $\overline D$, then a path between them in $C$ would contain an internal vertex of $D$ having degree at least two in $G[D,\overline D]$, contrary to the $1$-limited constraint.

Thus, every nontrivial component of $G[D,\overline D]$ contains exactly one vertex of $\overline D$. If this vertex has degree one, then the component is isomorphic to $K_2$. If it has degree two or three, then the component is isomorphic to $K_{1,2}$ or $K_{1,3}$, respectively, with center in $\overline D$. Since there are exactly $d_2$ vertices of degree two and $d_3$ vertices of degree three in $G[D,\overline D]$, this yields $d_2$ copies of $K_{1,2}$ and $d_3$ copies of $K_{1,3}$, while all remaining nontrivial components are isomorphic to $K_2$. Together with the $d_0$ isolated vertices, this proves~(iii).
\end{proof}

The decomposition given in Theorem~\ref{thm} can be translated directly into properties of the subgraphs induced by $D$ and $\overline D$.

\begin{corollary}\label{cor:0}
Under the assumptions of Theorem \ref{thm}, the following statements hold.
\begin{enumerate}
\item[$\mathrm{(i)}$] If $d_0=0$, then $G[D]$ is $2$-regular and $c_{\mathrm{ind}}(G)\geq\tfrac n2+t$.
\item[$\mathrm{(ii)}$] If $d_2=d_3=0$, then $G[\overline D]$ is $2$-regular and $c_{\mathrm{ind}}(G)\geq\tfrac n2-t$.
\end{enumerate}
\end{corollary}

\begin{proof}
If $d_0=0$, then every vertex of $D$ has degree one in $G[D,\overline D]$,  and hence degree two in $G[D]$. Thus, $G[D]$ is an induced $2$-regular subgraph of order $|D|=\frac n2+t$, which yields $c_{\mathrm{ind}}(G)\geq\tfrac n2+t$.

Similarly, from $d_2=d_3=0$, it follows that every vertex of $\overline D$ has degree one in $G[D,\overline D]$ and hence degree two in $G[\overline D]$. Therefore, $G[\overline D]$ is an induced $2$-regular subgraph of order $|\overline D|=\frac n2-t$ giving that $c_{\mathrm{ind}}(G)\geq\tfrac n2-t$.
\end{proof}

We next relate Corollary~\ref{cor:0} to the known extremal case. Recall that a \emph{matching} in a graph is a set of pairwise nonincident edges, and that a matching is \emph{perfect} if every vertex of the graph is incident with exactly one edge of the matching.

Fakhran et al.~\cite{fakhran} characterized the connected cubic graphs
satisfying $\gamma_{1,2}(G)=\frac n2$. They showed that these are precisely
the graphs consisting of two vertex-disjoint $2$-regular graphs of equal
order joined by a perfect matching. Since $\gammaenaL(G)=\gamma_{1,2}(G)$ for cubic graphs, their characterization can be stated in our terminology as follows.

\begin{proposition}
\label{n/2}
\cite{fakhran}
Let $G$ be a connected cubic graph of order $n$. Then
$\gammaenaL(G)=\frac n2$
if and only if there exists a set $D\subseteq V(G)$ such that
$E(G[D,\overline D])$
is a perfect matching of $G$.
\end{proposition}

It is therefore clear that a graph $G$ with $\gammaenaL(G)=\frac n2$, i.e., $t=0$, satisfies
$\cind(G)\geq\frac n2$.
This also follows immediately from Theorem~\ref{thm} and Corollary~\ref{cor:0}.

For $t>0$, the vertices counted by $d_0$ are precisely the obstruction
to $G[D]$ being $2$-regular: they have degree three in $G[D]$, whereas
all other vertices of $D$ have degree two. In some configurations this
obstruction can be eliminated by deleting a single vertex. The next
result characterizes exactly when this is possible.

\begin{proposition}\label{lem:-one}
Let $D$ be a $1$-limited dominating set of a cubic graph $G$. Then there exists a vertex $v\in D$ such that $G[D\setminus\{v\}]$ is $2$-regular if and only if one of the following statements holds:
\begin{enumerate}
\item[$\mathrm{(i)}$] $d_0=2$, and the two isolated vertices of $G[D,\overline D]$ have a common neighbor in $G[D]$;
\item[$\mathrm{(ii)}$] $d_0=4$, and one of the four isolated vertices of $G[D,\overline D]$ is adjacent in $G[D]$ to the other three.
\end{enumerate}
\end{proposition}

\begin{proof}
Let $D\subseteq V(G)$ be a $1$-limited dominating set of $G$ and let $S\subseteq V(G)$ be the set of isolated vertices of $G[D,\overline D]$, so that $|S|=d_0$. Recall that the vertices of $S$ have degree three in $G[D]$, whereas all other vertices of $D$ have degree two.

\medskip
Suppose first that there exists a vertex $v\in D$ such that $G[D\setminus\{v\}]$ is $2$-regular. Since every neighbor of $v$ in $G[D]$ loses one incident edge when $v$ is removed, each of these neighbors must belong to $S$. Consequently, $N_{G[D]}(v)\subseteq S$. Moreover, every vertex of $S\setminus\{v\}$ must be adjacent to $v$, as otherwise its degree would remain three in $G[D\setminus\{v\}]$.

If $v\notin S$, then $v$ has degree two in $G[D]$, and its two neighbors are precisely the vertices of $S$. Thus, $d_0=2$, and the two isolated vertices of $G[D,\overline D]$ have the common neighbor $v$, giving~(i).

Next, if $v\in S$, then $v$ has degree three in $G[D]$. Its three neighbors are precisely the remaining vertices of $S$, and therefore $d_0=4$. Hence, $v$ is adjacent in $G[D]$ to the other three isolated vertices of $G[D,\overline D]$, giving~(ii).

\medskip
Conversely, suppose first that~(i) holds, and let $v$ be a common neighbor of the two isolated vertices $u_1$ and $u_2$. Since $v$ is distinct from $u_1$ and $u_2$, it is not isolated in $G[D,\overline{D}]$, and therefore has degree two in $G[D]$. Hence, in $G[D\setminus\{v\}]$, the degrees of $u_1$ and $u_2$ are reduced from three to two, while all other remaining vertices retain degree two. Thus, $G[D\setminus\{v\}]$ is $2$-regular.

Finally, suppose that~(ii) holds, and let $v$ be the isolated vertex adjacent in $G[D]$ to the other three isolated vertices. Again, in $G[D\setminus\{v\}]$, the degrees of these three vertices are reduced from three to two, while all other remaining vertices retain degree two. Hence, $G[D\setminus\{v\}]$ is again $2$-regular.
\end{proof}

The configurations in Proposition~\ref{lem:-one} immediately yield the
conjectured lower bound for the induced cycles vertex number.

\begin{corollary}\label{cor}
Let $G$ be a cubic graph of order $n$, and let $D$ be a $1$-limited dominating set of $G$. If either condition $\mathrm{(i)}$ or $\mathrm{(ii)}$ of Proposition~\ref{lem:-one} holds, then $c_{\mathrm{ind}}(G)\geq\frac n2$.
\end{corollary}

\begin{proof}
Let $D\subseteq V(G)$ be a $1$-limited dominating set. By~\eqref{eq}, $|D|\geq\gammaenaL(G)\geq\frac n2$.
Hence we may write
$|D|=\frac n2+t$
for some nonnegative integer $t$.  Either of the conditions of Proposition \ref{lem:-one} implies that $d_0>0$, and hence Theorem \ref{thm}(i) gives $t\ge 1$. Moreover, Proposition~\ref{lem:-one} provides a vertex $v\in D$ such that $G[D\setminus\{v\}]$ is $2$-regular. Therefore, $c_{\mathrm{ind}}(G)\geq|D|-1
=\frac n2+t-1
\geq\frac n2
$.
\end{proof}

After applying these results to snarks in the next section, we return to them in Section~\ref{conclusion}, where we discuss their possible relevance to the conjectured bound for the induced cycles vertex number.

\section{Applications to snarks}

The structural results obtained in the previous section provide a convenient framework for studying $1$-limited domination in specific classes of cubic graphs. We now apply them to two classical families of snarks, generalized Blanu\v{s}a snarks and flower snarks.

\subsection{Blanu\v{s}a snarks}

The \emph{generalized Blanu\v{s}a snarks} are constructed from three basic blocks, denoted by $B$, $A_1$, and $A_2$, as shown in Figure~\ref{blan-bas}. The block $B$ has eight vertices, and  in the representation shown in the figure, its four corner vertices are labeled $a,b,c,d$. On the other hand, each of the blocks $A_1$ and $A_2$ has ten vertices, and its four corner vertices are labeled $x,y,z,w$.

Let $k\ge1$ and let $r\in\{1,2\}$. The generalized Blanu\v{s}a snark of type $r$, denoted by $B_k^r$, is constructed from $k$ copies of the block $B$ and one copy of the special block $A_r$, where $A_r=A_1$ if $r=1$, and $A_r=A_2$ if $r=2$. We denote the copies of $B$ by $B^1,B^2,\ldots,B^k$. For each $i\in[k]$, the corner vertices of $B^i$ corresponding to $a,b,c,d$ in $B$ are denoted by $a_i,b_i,c_i,d_i$, respectively, while the corner vertices of the single copy of $A_r$ retain without subscripts.

The blocks are arranged in the cyclic order $B^1,B^2,\ldots,B^k,A_r$. For every $i\in[k-1]$ we join $B^i$ to $B^{i+1}$ by connecting $b_i$ to $c_{i+1}$ and $d_i$ to $a_{i+1}$. We then join $B^k$ to $A_r$ by connecting $b_k$ to $z$ and $d_k$ to $x$, and close the cyclic sequence by connecting $y$ to $c_1$ and $w$ to $a_1$. The resulting graph $B_k^r$ is cubic and has $|V(B_k^r)|=8k+10$ vertices. 

\begin{figure}[!ht]
    \centering
    \begin{tikzpicture}[scale=0.9]

\tikzset{
    BlackVertex/.style={
        draw,
        circle,
        fill=black,
        minimum size=5pt,
        inner sep=0pt
    },
    WhiteVertex/.style={
        draw,
        circle,
        fill=white,
        minimum size=5pt,
        inner sep=0pt
    },
    GraphEdge/.style={
        black,
        thick
    },
    MatchingEdge/.style={
        red,
        very thick
    }
}

\begin{scope}[xshift=0cm]

\coordinate (Ba) at (0,3);
\coordinate (Bb) at (1,3);
\coordinate (Bc) at (2,3);
\coordinate (Bd) at (0,2);
\coordinate (Be) at (2,2);
\coordinate (Bf) at (0,1);
\coordinate (Bg) at (1,1);
\coordinate (Bh) at (2,1);

\draw[GraphEdge] (Ba) -- (Bb);
\draw[GraphEdge] (Bd) -- (Be);
\draw[GraphEdge] (Bg) -- (Bh);
\draw[GraphEdge] (Bd) -- (Bf);
\draw[GraphEdge] (Bb) -- (Bg);
\draw[GraphEdge] (Bc) -- (Be);
\draw[MatchingEdge] (Bb) -- (Bc);
\draw[MatchingEdge] (Bf) -- (Bg);
\draw[MatchingEdge] (Ba) -- (Bd);
\draw[MatchingEdge] (Be) -- (Bh);

\node[BlackVertex] at (Ba) {};
\node[BlackVertex] at (Bb) {};
\node[WhiteVertex] at (Bc) {};
\node[WhiteVertex] at (Bd) {};
\node[WhiteVertex] at (Be) {};
\node[WhiteVertex] at (Bf) {};
\node[BlackVertex] at (Bg) {};
\node[BlackVertex] at (Bh) {};

\node[above] at (1,3.25) {$B$};

\node[left] at (Ba) {$a$};
\node[right] at (Bc) {$b$};
\node[left] at (Bf) {$c$};
\node[right] at (Bh) {$d$};

\end{scope}

\begin{scope}[xshift=4.5cm]

\coordinate (Aonep) at (0,3);
\coordinate (Aoneq) at (0,1);
\coordinate (Aonea) at (1,3);
\coordinate (Aoneb) at (2,3);
\coordinate (Aonec) at (3,3);
\coordinate (Aoned) at (1,2);
\coordinate (Aonee) at (3,2);
\coordinate (Aonef) at (1,1);
\coordinate (Aoneg) at (2,1);
\coordinate (Aoneh) at (3,1);

\draw[GraphEdge] (Aonep) -- (Aonea);
\draw[GraphEdge] (Aonef) -- (Aoneq);
\draw[GraphEdge] (Aonea) -- (Aoneb);
\draw[GraphEdge] (Aoned) -- (Aonee);
\draw[GraphEdge] (Aoneg) -- (Aoneh);
\draw[GraphEdge] (Aoned) -- (Aonef);
\draw[GraphEdge] (Aoneb) -- (Aoneg);
\draw[GraphEdge] (Aonec) -- (Aonee);

\draw[MatchingEdge] (Aonep) -- (Aoneq);
\draw[MatchingEdge] (Aoneb) -- (Aonec);
\draw[MatchingEdge] (Aonef) -- (Aoneg);
\draw[MatchingEdge] (Aonea) -- (Aoned);
\draw[MatchingEdge] (Aonee) -- (Aoneh);

\node[BlackVertex] at (Aonep) {};
\node[WhiteVertex] at (Aoneq) {};
\node[BlackVertex] at (Aonea) {};
\node[BlackVertex] at (Aoneb) {};
\node[WhiteVertex] at (Aonec) {};
\node[WhiteVertex] at (Aoned) {};
\node[WhiteVertex] at (Aonee) {};
\node[WhiteVertex] at (Aonef) {};
\node[BlackVertex] at (Aoneg) {};
\node[BlackVertex] at (Aoneh) {};

\node[above] at (1.5,3.25) {$A_1$};

\node[left] at (Aonep) {$x$};
\node[left] at (Aoneq) {$z$};
\node[right] at (Aonec) {$y$};
\node[right] at (Aoneh) {$w$};

\end{scope}

\begin{scope}[xshift=9cm]

\coordinate (Atwoa) at (1,3);
\coordinate (Atwob) at (2,3);
\coordinate (Atwoc) at (3,3);
\coordinate (Atwod) at (1,2);
\coordinate (Atwoe) at (3,2);
\coordinate (Atwof) at (1,1);
\coordinate (Atwog) at (2,1);
\coordinate (Atwoh) at (3,1);
\coordinate (Atwop) at (2,2.5);
\coordinate (Atwoq) at (2.5,2);

\coordinate (Atwolefta) at (0.45,1.75);
\coordinate (Atwoleftb) at (0.45,2.05);
\coordinate (Atworightc) at (3.55,1.75);
\coordinate (Atworightd) at (3.55,2.05);


\draw[GraphEdge] (Atwoa) -- (Atwob);
\draw[GraphEdge] (Atwod) -- (Atwoe);
\draw[GraphEdge] (Atwog) -- (Atwoh);
\draw[GraphEdge] (Atwod) -- (Atwof);
\draw[GraphEdge] (Atwob) -- (Atwog);
\draw[GraphEdge] (Atwoc) -- (Atwoe);

\draw[MatchingEdge] (Atwob) -- (Atwoc);
\draw[MatchingEdge] (Atwof) -- (Atwog);
\draw[MatchingEdge] (Atwoa) -- (Atwod);
\draw[MatchingEdge] (Atwoe) -- (Atwoh);
\draw[MatchingEdge] (Atwop) -- (Atwoq);

\node[BlackVertex] at (Atwoa) {};
\node[BlackVertex] at (Atwob) {};
\node[WhiteVertex] at (Atwoc) {};
\node[WhiteVertex] at (Atwod) {};
\node[WhiteVertex] at (Atwoe) {};
\node[WhiteVertex] at (Atwof) {};
\node[BlackVertex] at (Atwog) {};
\node[BlackVertex] at (Atwoh) {};
\node[BlackVertex] at (Atwop) {};
\node[WhiteVertex] at (Atwoq) {};

\node[above] at (2,3.25) {$A_2$};

\node[left] at (Atwoa) {$x$};
\node[left] at (Atwof) {$z$};
\node[right] at (Atwoc) {$y$};
\node[right] at (Atwoh) {$w$};

\end{scope}

\end{tikzpicture}
    \caption{The basic blocks used in the construction of generalized Blanu\v{s}a snarks. Black vertices belong to $D$, white vertices belong to $\overline{D}$, and the red edges join vertices from the two sets.}
    \label{blan-bas}
\end{figure}
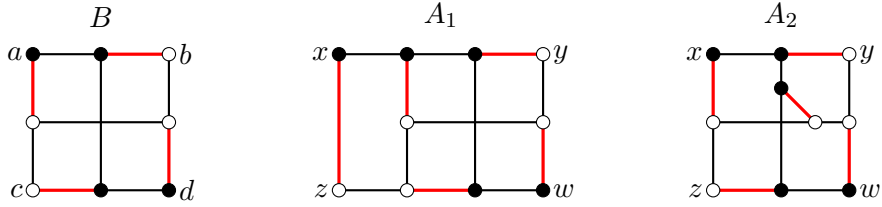

\begin{proposition}\label{thm:blanusa}
Let $k\ge1$ and let $r\in\{1,2\}$. Then
$$
   \gammaenaL(B_k^r)=4k+5.
$$
\end{proposition}

\begin{proof}
Let $D$ be the set of black vertices in the copies of the blocks $B_i$, $i\in[k]$, and $A_r$ shown in Figure~\ref{blan-bas}, and let $\overline D$ be the set of white vertices. Furthermore, recall that $E(B_k^r[D,\overline{D}])$ is the set of edges of $B_k^r$ with one endvertex in $D$ and the other in $\overline D$.

Note that within each block, the red edges have one endvertex in $D$ and the other in $\overline D$, whereas every black edge has both endvertices in the same set, see Figure~\ref{blan-bas}. Moreover, joining consecutive blocks introduces no additional edge between $D$ and $\overline{D}$, since both endvertices of such edges belong to the same color set $D$ or $\overline{D}$. Hence, $E(B_k^r[D,\overline{D}])$ consists precisely of the red edges shown in the copies of $B_i$ and $A_r$.

Every vertex of $B_k^r$ is incident with exactly one red edge, and therefore $B_k^r[D,\overline D]$ is a disjoint union of $4k+5$ copies of $K_2$. Since $B_k^r$ is cubic and has $8k+10$ vertices, Corollary~\ref{n/2} now yields
$$
\gammaenaL(B_k^r)
 =\frac{|V(B_k^r)|}{2}
 =\frac{8k+10}{2}
 =4k+5.
$$
\end{proof}

\subsection{Flower snarks}

A \textit{flower snark} $J_k$ is a cubic graph constructed as follows. Begin with $k \geq 3$ copies of the star graph $K_{1,3}$, also known as a \textit{claw}. In the $i$th copy, let $a_i$ denote the central vertex and $b_i$, $c_i$, $d_i$ its three leaves. The $b$-vertices are joined to form the $k$-cycle $b_1b_2\dots b_kb_1$, while the $c$- and $d$-vertices are joined to create the $2k$-cycle $c_1c_2\dots c_kd_1d_2\dots d_kc_1$.

Following the terminology used in the literature, in particular in~\cite{Burdett}, we refer to the graphs obtained by this construction as flower snarks $J_k$ for all $k\geq3$. With the definition of a snark adopted in this paper, the members of this family that are snarks are precisely those with odd $k\geq5$.

Figure~\ref{J7} shows two representations of $J_7$. The circular representation on the left is more common, whereas the linear representation on the right is more convenient for our proofs. Following the notation of \cite{Burdett}, we denote the $i$th claw by $J^i$ and arrange its vertices in the order $b_i,a_i,c_i,d_i$ from top to bottom. In the linear representation, consecutive claws are joined by $b_ib_{i+1}$, $c_ic_{i+1}$, and $d_id_{i+1}$ for $i\in[k-1]$, while the cyclic sequence is closed by the edges $b_kb_1$, $c_kd_1$, and $d_kc_1$. We refer to the two crossed edges $c_kd_1$ and $d_kc_1$ collectively as the \emph{twist}.

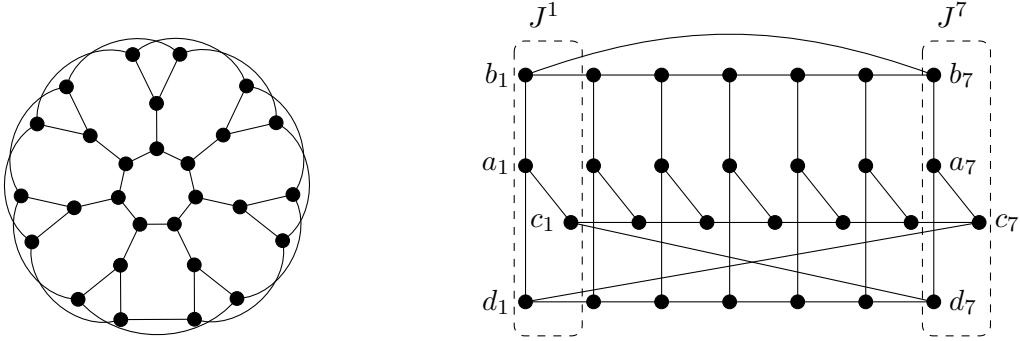
\begin{figure}[ht!]
\begin{center}

\begin{tikzpicture}[scale=1.5]

\tikzset{My Style/.style={draw, circle, fill=black, scale=0.5}}

\foreach \i in {1,...,7} {
    \coordinate (b\i) at ({296-360/7*\i}:0.35); 
}

\foreach \i in {1,...,7} {
    \coordinate (a\i) at ({296-360/7*\i}:0.75); 
}

\foreach \i in {1,...,7} {
    \coordinate (c\i) at ({296-360/7*\i+10}:1.2); 
    \coordinate (d\i) at ({296-360/7*\i-10}:1.2); 
}

\foreach \i in {1,...,7} {
    \draw (a\i) -- (b\i);
    \draw (a\i) -- (c\i);
    \draw (a\i) -- (d\i);
}

\foreach \i in {1,...,7} {
    \pgfmathtruncatemacro{\nexti}{mod(\i,7)+1}
    \draw (b\i) -- (b\nexti); 
}

\foreach \i in {1,...,6} { 
    \pgfmathtruncatemacro{\nexti}{mod(\i,7)+1}
    \draw[bend left=60] (c\i) to (c\nexti); 
}

\foreach \i in {1,...,6} { 
    \pgfmathtruncatemacro{\nexti}{mod(\i,7)+1}
    \draw[bend left=60] (d\i) to (d\nexti); 
}

\draw[bend left=50] (c7) to (d1); 
\draw[] (d7) to (c1); 

\foreach \i in {1,...,7} {
    \node[My Style] at (b\i) {};
    \node[My Style] at (c\i) {};
    \node[My Style] at (d\i) {};
    \node[My Style] at (a\i) {};

}



\foreach \i in {11,...,18} {
    \coordinate (b\i) at ({3.25+(\i-11)*0.6},1.8-0.8);
    \coordinate (a\i) at ({3.25+(\i-11)*0.6},1-0.8);
    \coordinate (c\i) at ({3.25+(\i-11)*0.6+0.4},0.5-0.8);
    \coordinate (d\i) at ({3.25+(\i-11)*0.6},-0.2-0.8);

    \coordinate (lb\i) at ({3.15+(\i-11)*0.6},1.8-0.8+0.3);
    \coordinate (ld\i) at ({3.15+(\i-11)*0.6},-0.2-0.8-0.3);
}

\foreach \i in {11,...,17} {
    \draw (a\i) -- (b\i);
    \draw (a\i) -- (c\i);
    \draw (a\i) -- (d\i);
}

\foreach \i in {11,...,16} {
    \pgfmathtruncatemacro{\nexti}{\i+1}
    \draw (b\i) -- (b\nexti);
    \draw (c\i) -- (c\nexti);
    \draw (d\i) -- (d\nexti);
}

\foreach \i in {11,...,17} {
    \node[My Style] at (b\i) {};
    \node[My Style] at (a\i) {};
    \node[My Style] at (c\i) {};
    \node[My Style] at (d\i) {};
}

    \node[left=2pt] at (b11) {\(b_{1}\)};
    \node[left=2pt] at (c11) {\(c_{1}\)};
    \node[left=2pt] at (d11) {\(d_{1}\)};
    \node[left=2pt] at (a11) {\(a_{1}\)};
    
    \node[right=2pt] at (b17) {\(b_{7}\)};
    \node[right=2pt] at (c17) {\(c_{7}\)};
    \node[right=2pt] at (d17) {\(d_{7}\)};
    \node[right=2pt] at (a17) {\(a_{7}\)};

    \draw[bend left=20] (b11) to (b17);

    \draw (d17) to (c11);
    \draw (d11) to (c17);

    \draw[rounded corners][dashed] (lb11) -- (ld11) -- (ld12) -- (lb12) -- cycle;
    \node[above right =2pt] at (lb11) {\(J^{1}\)};

    \draw[rounded corners][dashed] (lb17) -- (ld17) -- (ld18) -- (lb18) -- cycle;
    \node[above right =2pt] at (lb17) {\(J^{7}\)};
 
\end{tikzpicture}
\end{center}
\caption{Two visualizations of the flower snark $J_7$.}
\label{J7}
\end{figure}

It is important to note that the position of the twist is a feature of the chosen representation rather than of the underlying graph. To see this, suppose that the twist lies between $J^{i-1}$ and $J^i$, and redraw the claw $J^i$ by placing $c_i$ in the former position of $d_i$ and $d_i$ in the former position of $c_i$, without changing any adjacencies. The crossed connections between $J^{i-1}$ and $J^i$ then appear parallel, while the previously parallel connections between $J^i$ and $J^{i+1}$ become crossed. Thus, the twist moves from one side of $J^i$ to the other, although the underlying graph remains unchanged. Repeating this procedure shows that the twist may be placed between any pair of consecutive claws.

We conclude the description of the flower snark with a simple structural observation that will be useful later; the bipartiteness of $J_k$ is determined entirely by the parity of $k$.

\begin{lemma}\label{bipartite}
The flower snark $J_k$, $k\ge3$, is bipartite if and only if $k$ is even.
\end{lemma}

\begin{proof}
Suppose that $k$ is even, and define
$$
A=\bigcup_{i\equiv0(\mathrm{mod}\ 2)}
\{a_i,b_{i+1},c_{i+1},d_{i+1}\}
\qquad\text{and}\qquad
B=\bigcup_{i\equiv1(\mathrm{mod}\ 2)}
\{a_i,b_{i+1},c_{i+1},d_{i+1}\}.
$$
It is then clear that sets $A$ and $B$ form a partition of $V(J_k)$. Moreover, every edge within a claw and every edge joining consecutive claws has one endvertex in each set. Since $k$ is even, the same holds for the two edges forming the twist. Thus, both $A$ and $B$ are independent, and hence $J_k$ is bipartite.

Conversely, if $k$ is odd, then $b_1b_2\cdots b_kb_1$ is an odd cycle, and therefore $J_k$ cannot be bipartite.
\end{proof}

We now turn our attention to the $1$-limited domination number of flower snarks. We begin by establishing a constructive upper bound, which will later be shown to be sharp.

\begin{lemma}\label{upper}
Let $k\geq3$ and let $J_k$ be a flower snark. Then $\gammaenaL(J_k)\leq2k+1$, if $k$ is odd, and $\gammaenaL(J_k)\leq2k+2$, if $k$ is even.
\end{lemma}

\begin{proof}
Let $J_k$ be a flower snark with $k \geq 3$. To prove the stated upper bounds, it suffices to construct a $1$-limited dominating set $D$ of $J_k$ of the required cardinality. In what follows, we distinguish six cases according to the value of $k$ modulo $6$. For this purpose, we first introduce six building blocks $A$, $B$, $C$, and $X$, $Y$, $Z$ as shown in Figure~\ref{ZGkon}. 

To facilitate further explanation, we also introduce labels for the vertices within each of these blocks. 
For instance, the block $A$ consists of four vertices, $V(A) = \{b_1^A, a_1^A, c_1^A, d_1^A\}$,
while the block $B$ contains twelve vertices,
$V(B) = \{b_i^B, a_i^B, c_i^B, d_i^B \mid i \in [3]\}$;
the superscript identifies the type of block, while the subscript specifies the position of a vertex within that block. For example, the graph $J_9$ can be constructed from one copy of block $B$ and one copy of block $C$ by adding the edges $b_3^B b_1^C$, $c_3^B c_1^C$, $d_3^B d_1^C$, $b_6^C b_1^B$, $c_6^C d_1^B$, and $d_6^C c_1^B$. In this sense, we may therefore denote this decomposition by $J_9=BC$.

\begin{figure}[ht!]
\begin{center}
\begin{tikzpicture}[scale=1.25]
\tikzset{My Style1/.style={draw, circle, fill=black, scale=0.5}}
\tikzset{My Style2/.style={draw, circle, fill=white, scale=0.5}}

\node[font=\large\bfseries] at (0,1.2){$A$};

\draw[thick] (-0.2,0.7)--(0.2,0.7);
\draw[thick] (0.1,-0.5)--(0.5,-0.5);
\draw[thick] (-0.2,-1)--(0.2,-1);

\node [My Style1, name=b1]   at (0,0.7) {};	
\node [My Style1, name=a1]    at (0,0) {};
\node [My Style2, fill=white, name=c1]   at (0.3,-0.5) {};	
\node [My Style1, name=d1]   at (0,-1) {};	

\node[label=left:$b_1^{A}$]  at (-0.1,0.7) {};
\node[label=left:$a_1^{A}$]  at (0,0) {};	
\node[label=right:$c_1^{A}$]  at (0.35,-0.5) {};	
\node[label=left:$d_1^{A}$]  at (-0.1,-1) {};	

\draw[thick] (b1)--(a1)--(c1);
\draw[thick] (d1)--(a1);

\draw[thick] ($(d1) + (-0.11,-0.11)$) rectangle ($(d1) + (0.11,0.11)$);
\draw[thick] ($(b1) + (-0.11,-0.11)$) rectangle ($(b1) + (0.11,0.11)$);

\node[font=\large\bfseries] at (2.7,1.2){$B$};

\draw[thick] (1.8,0.7)--(3.6,0.7);
\draw[thick] (2.1,-0.5)--(3.9,-0.5);
\draw[thick] (1.8,-1)--(3.6,-1);

\node [My Style1, name=b2]   at (2,0.7) {};	
\node [My Style1, name=a2]    at (2,0) {};
\node [My Style2, fill=white, name=c2]   at (2.3,-0.5) {};	
\node [My Style1, name=d2]   at (2,-1) {};	

\node [My Style1, name=b3]   at (2.7,0.7) {};	
\node [My Style2, name=a3]    at (2.7,0) {};
\node [My Style2, name=c3]   at (3,-0.5) {};	
\node [My Style2, name=d3]   at (2.7,-1) {};

\node [My Style1, name=b4]   at (3.4,0.7) {};	
\node [My Style1, name=a4]    at (3.4,0) {};
\node [My Style1, name=c4]   at (3.7,-0.5) {};	
\node [My Style2, name=d4]   at (3.4,-1) {};	

\node[label=left:$d_1^{B}$]  at (1.9,-1) {};
\node[label=left:$c_3^{B}$]  at (4.45,-0.5) {};

\draw[thick] (b2)--(a2)--(c2);
\draw[thick] (d2)--(a2);
\draw[thick] (b3)--(a3)--(c3);
\draw[thick] (d3)--(a3);
\draw[thick] (b4)--(a4)--(c4);
\draw[thick] (d4)--(a4);

\draw[thick] ($(d2) + (-0.11,-0.11)$) rectangle ($(d2) + (0.11,0.11)$);
\draw[thick] ($(c4) + (-0.11,-0.11)$) rectangle ($(c4) + (0.11,0.11)$);

\node[font=\large\bfseries] at (7.15,1.2){$C$};

\draw[thick] (5.2,0.7)--(9.1,0.7);
\draw[thick] (5.5,-0.5)--(9.4,-0.5);
\draw[thick] (5.2,-1)--(9.1,-1);

\node [My Style2, name=b5]   at (5.4,0.7) {};	
\node [My Style2, name=a5]    at (5.4,0) {};
\node [My Style1, name=c5]   at (5.7,-0.5) {};	
\node [My Style2, name=d5]   at (5.4,-1) {};	

\node [My Style2, name=b6]   at (6.1,0.7) {};	
\node [My Style1, name=a6]    at (6.1,0) {};
\node [My Style1, name=c6]   at (6.4,-0.5) {};	
\node [My Style1, name=d6]   at (6.1,-1) {};

\node [My Style2, name=b7]   at (6.8,0.7) {};	
\node [My Style2, name=a7]    at (6.8,0) {};
\node [My Style2, name=c7]   at (7.1,-0.5) {};	
\node [My Style1, name=d7]   at (6.8,-1) {};	

\node [My Style1, name=b8]   at (7.5,0.7) {};	
\node [My Style1, name=a8]    at (7.5,0) {};
\node [My Style2, name=c8]   at (7.8,-0.5) {};	
\node [My Style1, name=d8]   at (7.5,-1) {};

\node [My Style1, name=b9]   at (8.2,0.7) {};	
\node [My Style2, name=a9]    at (8.2,0) {};
\node [My Style2, name=c9]   at (8.5,-0.5) {};	
\node [My Style2, name=d9]   at (8.2,-1) {};

\node [My Style1, name=b10]   at (8.9,0.7) {};	
\node [My Style1, name=a10]    at (8.9,0) {};
\node [My Style1, name=c10]   at (9.2,-0.5) {};	
\node [My Style2, name=d10]   at (8.9,-1) {};

\node[label=left:$b_1^{C}$]  at (5.35,0.7) {};
\node[label=left:$c_1^{C}$]  at (5.55,-0.5) {};
\node[label=right:$c_6^{C}$]  at (9.3,-0.5) {};

\draw[thick] (b5)--(a5)--(c5);
\draw[thick] (d5)--(a5);
\draw[thick] (b6)--(a6)--(c6);
\draw[thick] (d6)--(a6);
\draw[thick] (b7)--(a7)--(c7);
\draw[thick] (d7)--(a7);
\draw[thick] (b8)--(a8)--(c8);
\draw[thick] (d8)--(a8);
\draw[thick] (b9)--(a9)--(c9);
\draw[thick] (d9)--(a9);
\draw[thick] (b10)--(a10)--(c10);
\draw[thick] (d10)--(a10);

\draw[thick] ($(c5) + (-0.11,-0.11)$) rectangle ($(c5) + (0.11,0.11)$);
\draw[thick] ($(c10) + (-0.11,-0.11)$) rectangle ($(c10) + (0.11,0.11)$);
\draw[thick] ($(b5) + (-0.11,-0.11)$) rectangle ($(b5) + (0.11,0.11)$);

\node[font=\large\bfseries] at (0,-2){$X$};

\draw[thick] (-0.2,-2.5)--(0.2,-2.5);
\draw[thick] (0.1,-3.7)--(0.5,-3.7);
\draw[thick] (-0.2,-4.3)--(0.2,-4.3);

\node [My Style2, name=xb1]   at (0,-2.5) {};	
\node [My Style1, name=xa1]    at (0,-3.2) {};
\node [My Style1, name=xc1]   at (0.3,-3.7) {};	
\node [My Style1, name=xd1]   at (0,-4.3) {};	

\draw[thick] (xb1)--(xa1)--(xc1);
\draw[thick] (xd1)--(xa1);

\draw[thick] ($(xc1) + (-0.11,-0.11)$) rectangle ($(xc1) + (0.11,0.11)$);
\draw[thick] ($(xd1) + (-0.11,-0.11)$) rectangle ($(xd1) + (0.11,0.11)$);

\node[font=\large\bfseries] at (2.7,-2){$Y$};

\draw[thick] (1.1,-2.5)--(4.3,-2.5);
\draw[thick] (1.4,-3.7)--(4.6,-3.7);
\draw[thick] (1.1,-4.3)--(4.3,-4.3);

\node [My Style2, name=yb2]   at (1.3,-2.5) {};	
\node [My Style1, name=ya2]    at (1.3,-3.2) {};
\node [My Style1, name=yc2]   at (1.6,-3.7) {};	
\node [My Style1, name=yd2]   at (1.3,-4.3) {};	

\node [My Style2, name=yb3]   at (2,-2.5) {};	
\node [My Style2, name=ya3]    at (2,-3.2) {};
\node [My Style2, name=yc3]   at (2.3,-3.7) {};	
\node [My Style1, name=yd3]   at (2,-4.3) {};

\node [My Style1, name=yb4]   at (2.7,-2.5) {};	
\node [My Style1, name=ya4]    at (2.7,-3.2) {};
\node [My Style2, name=yc4]   at (3,-3.7) {};	
\node [My Style1, name=yd4]   at (2.7,-4.3) {};	

\node [My Style1, name=yb5]   at (3.4,-2.5) {};	
\node [My Style1, name=ya5]    at (3.4,-3.2) {};
\node [My Style2, name=yc5]   at (3.7,-3.7) {};	
\node [My Style1, name=yd5]   at (3.4,-4.3) {};	

\node [My Style2, name=yb6]   at (4.1,-2.5) {};	
\node [My Style2, name=ya6]    at (4.1,-3.2) {};
\node [My Style2, name=yc6]   at (4.4,-3.7) {};	
\node [My Style1, name=yd6]   at (4.1,-4.3) {};

\draw[thick] (yb2)--(ya2)--(yc2);
\draw[thick] (yd2)--(ya2);
\draw[thick] (yb3)--(ya3)--(yc3);
\draw[thick] (yd3)--(ya3);
\draw[thick] (yb4)--(ya4)--(yc4);
\draw[thick] (yd4)--(ya4);
\draw[thick] (yb5)--(ya5)--(yc5);
\draw[thick] (yd5)--(ya5);
\draw[thick] (yb6)--(ya6)--(yc6);
\draw[thick] (yd6)--(ya6);

\draw[thick] ($(yc2) + (-0.11,-0.11)$) rectangle ($(yc2) + (0.11,0.11)$);
\draw[thick] ($(yd6) + (-0.11,-0.11)$) rectangle ($(yd6) + (0.11,0.11)$);
\draw[thick] ($(yc6) + (-0.11,-0.11)$) rectangle ($(yc6) + (0.11,0.11)$);

\node[font=\large\bfseries] at (7.15,-2){$Z$};

\draw[thick] (5.2,-2.5)--(9.1,-2.5);
\draw[thick] (5.5,-3.7)--(9.4,-3.7);
\draw[thick] (5.2,-4.3)--(9.1,-4.3);

\node [My Style2, name=b5]   at (5.4,-2.5) {};	
\node [My Style1, name=a5]    at (5.4,-3.2) {};
\node [My Style1, name=c5]   at (5.7,-3.7) {};	
\node [My Style1, name=d5]   at (5.4,-4.3) {};	

\node [My Style2, name=b6]   at (6.1,-2.5) {};	
\node [My Style2, name=a6]    at (6.1,-3.2) {};
\node [My Style1, name=c6]   at (6.4,-3.7) {};	
\node [My Style2, name=d6]   at (6.1,-4.3) {};

\node [My Style1, name=b7]   at (6.8,-2.5) {};	
\node [My Style1, name=a7]    at (6.8,-3.2) {};
\node [My Style1, name=c7]   at (7.1,-3.7) {};	
\node [My Style2, name=d7]   at (6.8,-4.3) {};	

\node [My Style1, name=b8]   at (7.5,-2.5) {};	
\node [My Style2, name=a8]    at (7.5,-3.2) {};
\node [My Style2, name=c8]   at (7.8,-3.7) {};	
\node [My Style2, name=d8]   at (7.5,-4.3) {};

\node [My Style1, name=b9]   at (8.2,-2.5) {};	
\node [My Style1, name=a9]    at (8.2,-3.2) {};
\node [My Style2, name=c9]   at (8.5,-3.7) {};	
\node [My Style1, name=d9]   at (8.2,-4.3) {};

\node [My Style2, name=b10]   at (8.9,-2.5) {};	
\node [My Style2, name=a10]    at (8.9,-3.2) {};
\node [My Style2, name=c10]   at (9.2,-3.7) {};	
\node [My Style1, name=d10]   at (8.9,-4.3) {};

\draw[thick] (b5)--(a5)--(c5);
\draw[thick] (d5)--(a5);
\draw[thick] (b6)--(a6)--(c6);
\draw[thick] (d6)--(a6);
\draw[thick] (b7)--(a7)--(c7);
\draw[thick] (d7)--(a7);
\draw[thick] (b8)--(a8)--(c8);
\draw[thick] (d8)--(a8);
\draw[thick] (b9)--(a9)--(c9);
\draw[thick] (d9)--(a9);
\draw[thick] (b10)--(a10)--(c10);
\draw[thick] (d10)--(a10);

\draw[thick] ($(d5) + (-0.11,-0.11)$) rectangle ($(d5) + (0.11,0.11)$);
\draw[thick] ($(c10) + (-0.11,-0.11)$) rectangle ($(c10) + (0.11,0.11)$);
\draw[thick] ($(d10) + (-0.11,-0.11)$) rectangle ($(d10) + (0.11,0.11)$);

\end{tikzpicture}
\end{center}
\caption{Building blocks of $J_k$.}
\label{ZGkon}
\end{figure}
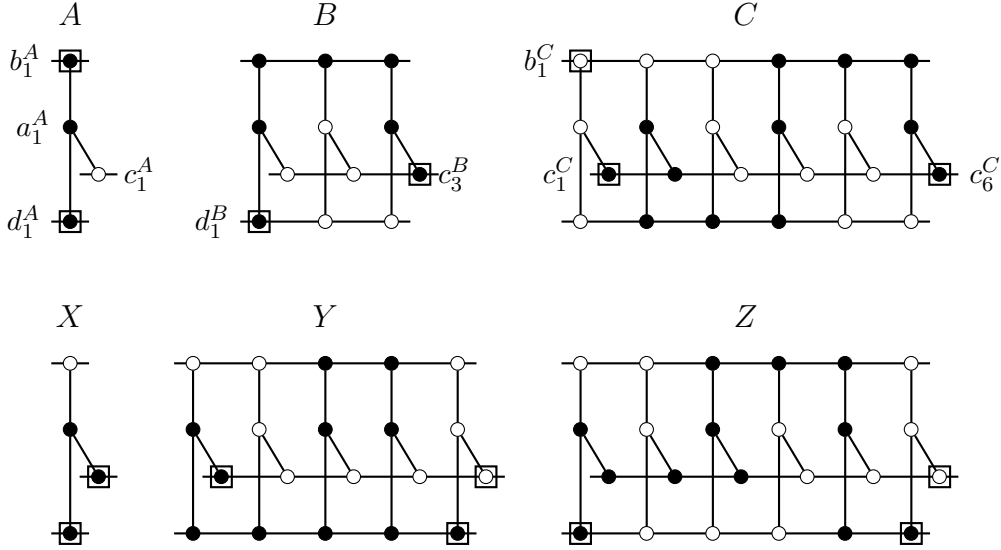

Let $D$ be the set of all black vertices in the blocks shown in Figure~\ref{ZGkon}. A black vertex marked with a square has exactly one black neighbor within its block, so it must have at least one black neighbor in an adjacent block in order to satisfy the $1$-limited constraint. An unmarked white vertex already has a black neighbor within its block, whereas a white vertex marked with a square has no black neighbor within its block and must therefore have a black neighbor in an adjacent block. 

These markings provide a quick way to verify that the set $D$ of black vertices is a $1$-limited dominating set. For instance, in the decomposition $J_9=BC$, the two vertices in each of the pairs $c_3^B,c_1^C$ and $c_6^C,d_1^B$ provide each other with the required black neighbor in the adjacent block, while $b_1^C$ has a black neighbor in $B$. It follows that the set of black vertices is a $1$-limited dominating set of $J_9$ of cardinality $19$.

More generally, if $k = 3 + 6m$, where $m \in \mathbb{N}_0$, then $J_k$ has the decomposition $BC^m$, in which the block $B$ is followed by $m$ consecutive copies of blocks $C$. The set of black vertices is a $1$-limited dominating set of cardinality $7 + 12m = 7 + 12 \cdot \tfrac{k - 3}{6} = 2k + 1$, as desired.  The remaining cases modulo $6$ are verified in the same manner, and the corresponding constructions are summarized in Table~\ref{tabela}. For odd $k$, the table provides a $1$-limited dominating set of cardinality $2k+1$, while for even $k$, it provides one of cardinality $2k+2$. This completes the proof. \end{proof} 

\begin{table}[h!]
\centering
\begin{tabular}{|c|c|c|c|}
\hline
$k$ ($m\in \mathbb{N}_0$) & \begin{tabular}{c} Structure  \\ of $J_k$ \end{tabular}  & \begin{tabular}{c} Cardinality of $D$ \\ in terms of $m$ \end{tabular} & \begin{tabular}{c} Cardinality of $D$ \\ in terms of $k$ \end{tabular}  \\
\hline
$3+6m$ & $BC^m$ & $7+12m$ & $2k+1$ \\
$4+6m$ & $ABC^m$ & $10+12m$ & $2k+2$\\
$5+6m$ & $YZ^m$ & $11+12m$ & $2k+1$\\
$6+6m$ & $XYZ^m$ & $14+12m$ & $2k+2$\\
$7+6m$ & $XZ^{m+1}$ & $15+12m$ & $2k+1$\\
$8+6m$ & $X^2Z^{m+1}$ & $18+12m$ & $2k+2$\\
\hline
\end{tabular}
\caption{Block decompositions of $J_k$ and the cardinalities of the corresponding $1$-limited dominating sets.}
\label{tabela}
\end{table}

We next turn our attention to the corresponding lower-bound argument. Through several auxiliary results, we show that the general lower bound~\eqref{eq} is not attained by flower snarks.

First we introduce some additional notation. For a set $S\subseteq V(J_k)$ and every $i\in[k]$, let $M_i^-$ and $M_i^+$ denote the sets of edges of $E(J_k[S,\overline S])$ joining $J^{i-1}$ to $J^i$ and $J^i$ to $J^{i+1}$, respectively, and set
$M_i^\pm=M_i^-\cup M_i^+$.

Since the twist may interchange the $c$- and $d$-connections, we also introduce notation for the neighbors of $b_i$, $c_i$, and $d_i$ in the adjacent claws.
For every $x_i\in\{b_i,c_i,d_i\}$, let $x_i^-$ and $x_i^+$ denote the unique neighbor of $x_i$ in $J^{i-1}$ and $J^{i+1}$, respectively. These symbols always refer to the actual neighbors of $x_i$, hence, at the twist, $x_i^-$ or $x_i^+$ need not carry the same letter as $x_i$. All indices are taken modulo $k$, so that $J^0=J^k$ and $J^{k+1}=J^1$.

In what follows, for every set $S\subseteq V(J_k)$, let
$M_S=E(J_k[S,\overline S])$.
We first consider the case when $M_S$ is a perfect matching.

\begin{lemma}\label{lem:ax-in-M}
Let $J_k$, $k\geq3$, be a flower snark, and let $S\subseteq V(J_k)$ be such that $M_S$ is a perfect matching. Let $i\in[k]$ and let
$\{x_i,y_i,z_i\}=\{b_i,c_i,d_i\}$.
If $a_ix_i\in M_S$, then 
$M_i^\pm=\{y_i^-y_i,z_iz_i^+\}$
or
$M_i^\pm=\{y_iy_i^+,z_i^-z_i\}$.
In particular,
$|M_i^-|=|M_i^+|=1$.
\end{lemma}

\begin{proof}
Let $T\in\{S,\overline S\}$ be the set containing $a_i$. Since
$a_ix_i\in M_S$, we have $x_i\in\overline T$. As $M_S$ is a perfect
matching, neither $a_iy_i$ nor $a_iz_i$ belongs to $M_S$, and hence
$y_i,z_i\in T$.
Since each of $y_i$ and $z_i$ must be incident with exactly one edge
of $M_S$, we have
\[
|M_S\cap\{y_i^-y_i,y_iy_i^+\}|=1
\qquad\text{and}\qquad
|M_S\cap\{z_i^-z_i,z_iz_i^+\}|=1.
\]
It remains to show that these two edges join $J^i$ to different
adjacent claws. 

Suppose, to the contrary, that both matching edges join $J^i$ to $J^{i+1}$. Then
$y_i^+,z_i^+\in\overline T$. Moreover, since $x_i$ is already matched with $a_i$, we have $x_ix_i^+\notin M_S$, and hence $x_i^+\in\overline T$. Thus all three vertices of
$\{b_{i+1},c_{i+1},d_{i+1}\}$ belong to $\overline T$. Consequently, $a_{i+1}$ is incident with either no edge or three edges of $M_S$, a contradiction.

The case in which both edges join $J^i$ to $J^{i-1}$ is analogous. 
Therefore,
$M_i^\pm=\{y_i^-y_i,z_iz_i^+\}$
or $M_i^\pm=\{y_iy_i^+,z_i^-z_i\}$.
In either case, $|M_i^-|=|M_i^+|=1$.
\end{proof}

If $S\subseteq V(J_k)$ is such that $M_S$ is a perfect matching, then the $a$- and $b$-vertices also follow a periodic pattern around the cyclic sequence of claws. More precisely, the $a$-vertices alternate between $S$ and $\overline S$, while the edges $a_ib_i$ belonging to $M_S$ occur at every third claw.

\begin{lemma}\label{lem:a-b}
Let $J_k$, $k\geq3$, be a flower snark, and let $S\subseteq V(J_k)$ be such that $M_S$ is a perfect matching. Then the following statements hold.
\begin{enumerate}
    \item[(i)] For every $i\in[k]$ and every $T\in\{S,\overline S\}$, if $a_i\in T$, then $a_{i+1}\in\overline T$.
    \item[(ii)] There exists some $i\in[k]$ such that $a_ib_i\in M_S$.
    \item[(iii)] For every $i\in[k]$, if $a_ib_i\in M_S$, then neither $a_{i+1}b_{i+1}$ nor $a_{i+2}b_{i+2}$ belongs to $M_S$, while $a_{i+3}b_{i+3}\in M_S$.
\end{enumerate}
Consequently, $k$ is divisible by $6$.
\end{lemma}

\begin{proof}

\medskip
(i) Let $T\in\{S,\overline S\}$ and suppose that $a_i\in T$. Since $M_S$ is a perfect matching, exactly one of $b_i,c_i,d_i$ belongs to $\overline T$; let $x_i$ denote this vertex, and let $y_i$ and $z_i$ denote the remaining two. Thus $a_ix_i\in M_S$, while $y_i,z_i\in T$.

Since $x_i$ is already matched with $a_i$, we have $x_ix_i^+\notin M_S$, and hence $x_i^+\in\overline T$. By Lemma~\ref{lem:ax-in-M}, exactly one of $y_i$ and $z_i$ is matched to its neighbor in $J^{i+1}$. Consequently, one of $y_i^+,z_i^+$ belongs to $\overline T$, while the other belongs to $T$. Thus exactly two of $b_{i+1},c_{i+1},d_{i+1}$ belong to $\overline T$, while the remaining one belongs to $T$. Since $M_S$ is a perfect matching, it follows that $a_{i+1}\in\overline T$.

\medskip
(ii) Suppose, to the contrary, that $a_ib_i\notin M_S$ for every $i\in[k]$. Then $a_i$ and $b_i$ belong to the same set for every $i\in[k]$. By~(i), consecutive $a$-vertices belong to complementary sets, and hence the same is true for consecutive $b$-vertices. Hence, $b_ib_{i+1}\in M_S$ for every $i\in[k]$, which would make every $b_i$ incident with two edges of $M_S$, a contradiction.

\medskip
(iii) Suppose that $a_ib_i\in M_S$ for some $i\in[k]$, and choose $T\in\{S,\overline S\}$ so that $a_i\in T$ and $b_i\in\overline T$. Since $b_i$ is already matched with $a_i$, we have $b_ib_{i+1}\notin M_S$, and hence $b_{i+1}\in\overline T$. By~(i), we also have $a_{i+1}\in\overline T$. Thus $a_{i+1}b_{i+1}\notin M_S$.
Since $M_S$ is a perfect matching, it follows that
$b_{i+1}b_{i+2}\in M_S$,
and hence $b_{i+2}\in T$. Since $b_{i+2}$ is already matched with $b_{i+1}$, we have
$b_{i+2}b_{i+3}\notin M_S$,
and therefore $b_{i+3}\in T$. By~(i), $a_{i+2}\in T$
and $a_{i+3}\in\overline T$.
Thus neither $a_{i+1}b_{i+1}$ nor $a_{i+2}b_{i+2}$ belongs to $M_S$, whereas $a_{i+3}b_{i+3}\in M_S$.

\medskip
Finally, property (i) shows that the $a$-vertices alternate between $S$ and $\overline S$ around the cyclic sequence of claws, so $k$ is even. On the other hand, properties~(ii) and~(iii) imply that $k$ is divisible by $3$. Consequently, $k$ is divisible by $6$.
\end{proof}

The preceding two lemmas show that, if $M_S$ is a perfect matching for some $S\subseteq V(J_k)$, the edges of $M_S$ are subject to strong local and periodic restrictions. We now use these restrictions to show that no such perfect matching exists for any flower snark $J_k$.

\begin{lemma}\label{lem:NO-S}
Let $J_k$, $k\geq3$, be a flower snark. Then there is no set
$S\subseteq V(J_k)$ such that $M_S$ is a perfect matching.
\end{lemma}

\begin{proof}
Suppose, to the contrary, that $M_S$ is a perfect matching for some
$S\subseteq V(J_k)$. By Lemma~\ref{lem:a-b}, $k=6m$ for some
$m\in\mathbb N$.

By Lemma~\ref{lem:a-b}(ii), we may choose $i\in[k]$ such that
$a_{i+1}b_{i+1}\in M_S$. Without loss of generality, assume that
$a_{i+1}\in\overline S$ and $b_{i+1}\in S$. Lemma~\ref{lem:a-b}(i)
then gives
\begin{equation}\label{a:inANDout}
a_{i+2t}\in S
\qquad\text{and}\qquad
a_{i+1+2t}\in\overline S,
\end{equation}
for every nonnegative integer $t$, where all indices are taken modulo
$k$. Moreover, Lemma~\ref{lem:a-b}(iii) shows that
$a_{i+1+3t}b_{i+1+3t}\in M_S$, while no other edge of the form
$a_jb_j$ belongs to $M_S$. It follows that
\begin{equation}\label{b:inANDout}
b_{i+6t},b_{i+1+6t},b_{i+2+6t}\in S
\qquad\text{and}\qquad
b_{i+3+6t},b_{i+4+6t},b_{i+5+6t}\in\overline S.
\end{equation}

It remains to determine which of the $c$- and $d$-vertices belong to
$S$ and which belong to $\overline S$. Before doing so, recall that
interchanging the positions of the $c$- and $d$-vertices throughout a
suitable consecutive sequence of claws moves the twist without
affecting the $a$- and $b$-vertices. We may therefore choose a
representation in which the twist lies between $J^{i-1}$ and $J^i$.
Consequently, all other connections between consecutive claws are
untwisted.

Since $a_{i+1}\in\overline S$, $b_{i+1}\in S$, and $M_S$ is a perfect
matching, we have
$c_{i+1},d_{i+1}\in\overline S$.
By Lemma~\ref{lem:ax-in-M}, one of the edges of $M_S$ incident with
$c_{i+1}$ and $d_{i+1}$ joins $J^{i+1}$ to $J^i$, while the other joins
$J^{i+1}$ to $J^{i+2}$. By globally interchanging the labels $c$ and
$d$, which does not change the position of the twist, we may assume that $c_ic_{i+1},\;d_{i+1}d_{i+2}\in M_S$.
Hence,
$c_i,d_{i+2}\in S$
and
$d_i,c_{i+2}\in\overline S$.
Since $a_{i+2}\in S$ and $c_{i+2}\in\overline S$, we have
$a_{i+2}c_{i+2}\in M_S$. As $M_S$ is a perfect matching,
$c_{i+2}c_{i+3}\notin M_S$, and hence
$c_{i+3}\in\overline S$.
Since $a_{i+3},b_{i+3},c_{i+3}\in\overline S$ and $M_S$ is a perfect
matching, we have $d_{i+3}\in S$. Moreover, since
$a_{i+4}\in S$, $b_{i+4}\in\overline S$, and $M_S$ is a perfect
matching, we have
$c_{i+4},d_{i+4}\in S$.
In particular,
$c_{i+3}c_{i+4}\in M_S$.
Since $a_{i+4}b_{i+4}\in M_S$, Lemma~\ref{lem:ax-in-M} gives
\[
M_{i+4}^{\pm}
=
\{c_{i+3}c_{i+4},d_{i+4}d_{i+5}\}.
\]
Thus $d_{i+5}\in\overline S$. Moreover, since $c_{i+4}$ is already
matched with $c_{i+3}$, we have $c_{i+4}c_{i+5}\notin M_S$, and hence
$c_{i+5}\in S$. Since $a_{i+5}\in\overline S$, it also follows that
$a_{i+5}c_{i+5}\in M_S$.
The resulting forced configuration is shown in Figure~\ref{fig:2}.

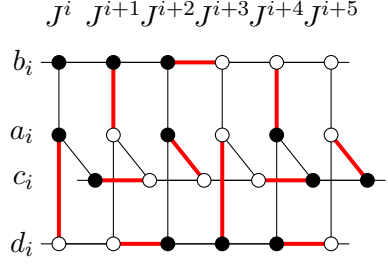
\begin{figure} [h!]
\centering

\begin{tikzpicture}[scale=1.2]

\tikzset{StyleBlack/.style={draw, circle, fill=black, scale=0.5}}
\tikzset{StyleWhite/.style={draw, circle, fill=white, scale=0.5}}
\tikzset{StyleGray/.style={draw, circle, fill=gray, scale=0.75}}

\foreach \i in {0,...,6} {
    \coordinate (b\i) at ({3.25+(\i-11)*0.6},1.8);
    \coordinate (a\i) at ({3.25+(\i-11)*0.6},1);
    \coordinate (c\i) at ({3.25+(\i-11)*0.6+0.4},0.5);
    \coordinate (d\i) at ({3.25+(\i-11)*0.6},-0.2);
}

\coordinate (bl0) at (3.15-11*0.6,1.8);
\coordinate (cl0) at (3.15-11*0.6,0.5);
\coordinate (dl0) at (3.15-11*0.6,-0.2);
\coordinate (al0) at (3.15-11*0.6,1);

\coordinate (alevo) at (3.15-11*0.6-0.1,1);
\coordinate (adesno) at (3.15-5.5*0.6,1);
\coordinate (blevo) at (3.15-11*0.6-0.1,1.8);
\coordinate (bdesno) at (3.15-5.5*0.6,1.8);
\coordinate (clevo) at (3.15-11*0.6-0.1+0.4,0.5);
\coordinate (cdesno) at (3.15-5.5*0.6+0.4,0.5);
\coordinate (dlevo) at (3.15-11*0.6-0.1,-0.2);
\coordinate (ddesno) at (3.15-5.5*0.6,-0.2);

\foreach \i in {0,...,5} {
    \draw (a\i) -- (b\i);
    \draw (a\i) -- (c\i);
    \draw (a\i) -- (d\i);
}

\draw (blevo) -- (bdesno);
\draw (clevo) -- (cdesno);
\draw (dlevo) -- (ddesno);

\foreach \i in {0,...,4} {
    \pgfmathtruncatemacro{\nexti}{\i+1}
    \draw (b\i) -- (b\nexti);
    \draw (c\i) -- (c\nexti);
    \draw (d\i) -- (d\nexti);
}

    \draw[red,line width=1.5pt] (a1) -- (b1);
    \draw[red,line width=1.5pt] (a4) -- (b4);
    \draw[red,line width=1.5pt] (b2) -- (b3);

    \draw[red,line width=1.5pt] (c0) -- (c1);
    \draw[red,line width=1.5pt] (c3) -- (c4);
    \draw[red,line width=1.5pt] (d1) -- (d2);
    \draw[red,line width=1.5pt] (d4) -- (d5);

    \draw[red,line width=1.5pt] (a0) -- (d0);
    \draw[red,line width=1.5pt] (a3) -- (d3);

    \draw[red,line width=1.5pt] (a2) -- (c2);
    \draw[red,line width=1.5pt] (a5) -- (c5);

    \node[StyleBlack] at (b0) {};
    \node[StyleBlack] at (b1) {};
    \node[StyleBlack] at (b2) {};
    \node[StyleWhite] at (b3) {};
    \node[StyleWhite] at (b4) {};
    \node[StyleWhite] at (b5) {};

    \node[StyleBlack] at (a0) {};
    \node[StyleBlack] at (a2) {};
    \node[StyleBlack] at (a4) {};
    \node[StyleWhite] at (a1) {};
    \node[StyleWhite] at (a3) {};
    \node[StyleWhite] at (a5) {};

    \node[StyleBlack] at (c0) {};
    \node[StyleWhite] at (c1) {};
    \node[StyleWhite] at (c2) {};
    \node[StyleWhite] at (c3) {};
    \node[StyleBlack] at (c4) {};
    \node[StyleBlack] at (c5) {};

    \node[StyleWhite] at (d0) {};
    \node[StyleWhite] at (d1) {};
    \node[StyleBlack] at (d2) {};
    \node[StyleBlack] at (d3) {};
    \node[StyleBlack] at (d4) {};
    \node[StyleWhite] at (d5) {};

\node[left=2pt] at (bl0) {$b_i$};
\node[left=2pt] at (al0) {$a_i$};
\node[left=2pt] at (cl0) {$c_i$};
\node[left=2pt] at (dl0) {$d_i$};

\def\i{0} 

\foreach \offset in {0,...,5} {
    \pgfmathtruncatemacro{\nodei}{\i + \offset }
    \node[above=10pt] at (b\nodei) {%
        \(
        J^{%
            i
            \ifnum\offset<0
                -\the\numexpr -\offset\relax
            \else
                \ifnum\offset>0
                    +\offset
                \fi
            \fi
        }
        \)
    };
}

\end{tikzpicture}
\caption{The forced configuration; black vertices belong to $S$, white vertices belong to $\overline{S}$, and red edges belong to $M_S$.}
\label{fig:2}
\end{figure}

If $k=6$, then the twist joins $J^{i+5}$ to $J^i$. More precisely,
$c_i\in S$ is adjacent to both $c_{i+1}\in\overline S$ and
$d_{i+5}\in\overline S$. Thus $c_i$ is incident with two edges of
$M_S$, contradicting the fact that $M_S$ is a perfect matching.

We may therefore assume that $k=6m$ for some integer $m\geq2$. Since
the twist lies between $J^{i-1}$ and $J^i$, all other connections
between consecutive claws are untwisted. By~\eqref{a:inANDout}
and~\eqref{b:inANDout}, we have
$a_{i+6},b_{i+6}\in S$.
Moreover, since $a_{i+5}c_{i+5}\in M_S$, the vertex $c_{i+5}$ is
already matched with $a_{i+5}$. Hence
$c_{i+5}c_{i+6}\notin M_S$, and therefore $c_{i+6}\in S$. Since
$a_{i+6},b_{i+6},c_{i+6}\in S$ and $M_S$ is a perfect matching, it
follows that
$d_{i+6}\in\overline S$.
Thus, the configuration on $J^{i+6}$ coincides with that on $J^i$.

Repeating the same arguments block by block shows that, for every
$t\in\{0,1,\ldots,m-1\}$, the configurations on
$J^{i+6t},J^{i+6t+1},\ldots,J^{i+6t+5}$
coincide with those on
$J^i,J^{i+1},\ldots,J^{i+5}$,
respectively. In particular,
$d_{i+k-1}\in\overline S$. At the twist, $c_i\in S$ is adjacent to
$d_{i+k-1}\in\overline S$, while it is also adjacent to
$c_{i+1}\in\overline S$. Hence $c_i$ is incident with two edges of
$M_S$, again contradicting the fact that $M_S$ is a perfect matching.
\end{proof}

The upper bounds established earlier, together with the structural restrictions obtained above, now determine the exact $1$-limited domination number of every flower snark.

\begin{theorem}\label{thm:Jk}
Let $k\geq3$, and let $J_k$ be a flower snark. Then
$$
\gamma_1^{\mathrm L}(J_k)=
    \begin{cases}
        2k+1, & \text{if $k$ is odd},\\[1mm]
        2k+2, & \text{if $k$ is even}.
    \end{cases}
$$
\end{theorem}

\begin{proof}
Since $|V(J_k)|=4k$,~\eqref{eq} gives $\gammaenaL(J_k)\geq2k$. Moreover, Lemma \ref{lem:NO-S} and Proposition \ref{n/2} rule out equality, and hence $\gammaenaL(J_k)\ge2k+1$.
On the other hand, Lemma \ref{upper} gives that $\gammaenaL(J_k)\le 2k+1$ when $k$ is odd. Consequently, $\gammaenaL(J_k)=2k+1$ whenever $k$ is odd. 

Assume now that $k$ is even, and suppose that
$\gammaenaL(J_k)=2k+1$, and let $D$ be a $\gammaenaL(J_k)$-set.
Since  $|D|=2k+1$, from Theorem \ref{thm}(i) and (ii), it follows that $d_0+d_2+2d_3=2$, where $d_0$ and $d_2$ are even. Consequently, $(d_0,d_2,d_3)\in\{(2,0,0), (0,0,1), (0,2,0)\}$. If $(d_0,d_2,d_3)=(2,0,0)$, then Corollary \ref{cor:0}(ii) shows that $J_k[\overline{D}]$ is $2$-regular of order $|\overline{D}|=2k-1$. In either of the other two cases, $d_0=0$, and Corollary \ref{cor:0}(i) ensures that $J_k[D]$ is $2$-regular of order $|D|=2k+1$. 
Thus, in each case, $J_k$ contains an induced $2$-regular subgraph of odd order, and hence an odd cycle. This contradicts Lemma~\ref{bipartite}, since $J_k$ is bipartite when $k$ is even. Therefore,
$\gammaenaL(J_k)\geq 2k+2$.
Together with the upper bound from Lemma~\ref{upper}, this yields
$\gammaenaL(J_k)=2k+2$ for even $k$.

\end{proof}

\section{Conclusion}
\label{conclusion}

We have established structural properties of minimum $1$-limited dominating sets in cubic graphs in terms of the parameter $t$ defined by
$\gammaenaL(G)=\frac n2+t$,
where $\tfrac n2$ is the lower bound in~\eqref{eq}. Our applications show that values close to this bound occur in classical infinite families of snarks. Generalized Blanu\v{s}a snarks attain the bound $\tfrac n2$, whereas odd and even flower snarks have $1$-limited domination numbers $\tfrac n2+1$ and $\tfrac n2+2$, respectively.

The relationship between $\gammaenaL(G)$ and the induced cycles vertex number $\cind(G)$ was recently investigated in~\cite{Trilobites}. There, the trilobite graphs were introduced as a family satisfying
$\gammaenaL(G)>\cind(G)$,
showing that the inequality $\gammaenaL(G)\leq\cind(G)$ does not hold in general. Two conjectures were proposed to describe how exceptional this phenomenon might be, but both were subsequently disproved in~\cite{Riste}. In particular, an infinite family $H(k)$ was constructed for which
$\cind(H(k))=\frac n2+2$
and
$\gammaenaL(H(k))=\frac n2+3$.
Nevertheless, all these examples satisfy the conjectured bound~\eqref{eq:cind-conj}.

The structural results obtained here suggest an alternative approach to this conjecture. Since
$\gammaenaL(G)=\frac n2+t$,
the bound~\eqref{eq:cind-conj} is equivalently written as
\[
\cind(G)\geq\gammaenaL(G)-t.
\]
Although this is only a reformulation, it is naturally suited to the structural viewpoint developed in this paper. Indeed, Theorem~\ref{thm} gives
$d_0+d_2+2d_3=2t$,
so that small values of $t$ severely restrict the number and types of vertices in $G[D,\overline D]$ whose degree differs from one.
For $t=0$, Theorem~\ref{thm} forces
$d_0=d_2=d_3=0$, and hence $G[D,\overline D]$ is a perfect matching. In this case, Corollary~\ref{cor:0} immediately yields
$\cind(G)\geq\frac n2$.
More generally, Corollary~\ref{cor:0}(i) establishes~\eqref{eq:cind-conj} whenever $d_0=0$, while Corollary~\ref{cor} covers certain additional configurations in which deleting a single vertex from $D$ yields an induced $2$-regular subgraph of order at least $\frac n2$.

The first case not completely covered by these results occurs when $\gammaenaL(G)=\frac n2+1$.
If $D$ is a $\gammaenaL(G)$-set, then $t=1$, and Theorem~\ref{thm}, together with the parity of $d_0$ and $d_2$, gives
\[
(d_0,d_2,d_3)\in
\{(0,2,0),(0,0,1),(2,0,0)\}.
\]
The first two cases satisfy~\eqref{eq:cind-conj} by Corollary~\ref{cor:0}(i). In the remaining case,
$(d_0,d_2,d_3)=(2,0,0)$,
let $x$ and $y$ be the two isolated vertices of $G[D,\overline D]$. If $x$ and $y$ have a common neighbor in $G[D]$, then Proposition~\ref{lem:-one}(i) again yields
$\cind(G)\geq\frac n2$.
Consequently, for $t=1$, the only configuration not resolved by the present results is the case in which $(d_0,d_2,d_3)=(2,0,0)$ and $x$ and $y$ have no common neighbor in $G[D]$.
In this remaining configuration, $G[\overline D]$ is $2$-regular but has only $\tfrac n2-1$ vertices, while Proposition~\ref{lem:-one} rules out obtaining an induced $2$-regular subgraph of order $\tfrac n2$ from $D$ by deleting a single vertex. Thus, any further argument along this structural approach must combine vertices from both $D$ and $\overline D$. Whether such an induced $2$-regular subgraph necessarily exists remains open.

\medskip

Finally, we record one further observation concerning the induced cycles vertex number. It is independent of the results above, but resolves a conjecture of Henning et al.~\cite{Henning} on the computational complexity of the problem.

Henning, Joos, L\"owenstein, and Sasse~\cite{Henning} considered the decision problem \textsc{Induced Cycle}, which asks, for a given graph $G$ and a positive integer $k$, whether
\[
\cind(G)\geq k.
\]
They proved that \textsc{Induced Cycle} is NP-complete for graphs of maximum degree~$4$ and conjectured that the problem remains NP-complete when restricted to cubic graphs.

The validity of this conjecture is in fact implied by an earlier result of Vattani et al.~\cite{Vattani}; we have verified Lemma 12 in the unpublished extended version of their paper
\footnote{The extended version of the paper, entitled \emph{Preserving Pairwise Relationships in Subgraphs}, is available at \url{https://faculty.mccombs.utexas.edu/deepayan.chakrabarti/mywww/papers/icml11-ppvpreserving.pdf}.},
in which they prove that finding a largest induced $2$-regular subgraph is NP-hard even for cubic graphs. Since $\cind(G)$ is precisely the maximum number of vertices in an induced $2$-regular subgraph of $G$, this yields the NP-hardness of \textsc{Induced Cycle} on cubic graphs.

Moreover, \textsc{Induced Cycle} belongs to NP, since for a given vertex set $U\subseteq V(G)$ one can verify in polynomial time that $|U|\geq k$ and that $G[U]$ is $2$-regular. Consequently, the conjecture of Henning et al.~is resolved.

\begin{corollary}
\label{cor:induced-cycle-NP-complete}
\textsc{Induced Cycle} is NP-complete for cubic graphs.
\end{corollary}

\vskip 1pc

\noindent{\bf Acknowledgments.}
The second author acknowledges the financial support of the Slovenian Research and Innovation Agency through research core funding No.\ P1-0288. The third author acknowledges the financial support of the Slovenian Research and Innovation Agency through research core funding No.\ P1-0297 and project No.\ J1-70016.
This work has also been supported by the European Commission's Horizon Europe Research and Innovation programme through the Marie Sk\l odowska-Curie Actions Staff Exchanges (MSCA-SE) under Grant Agreement no. 101182819 (COVER: (C)ombinatorial (O)ptimization for (V)ersatile Applications to (E)merging u(R)ban Problems).

\vskip 1pc

\noindent{\bf Statements and Declarations.}
The authors have no conflict of interest to declare.


\end{document}